\documentclass[11pt,reqno]{amsart}
\usepackage{amssymb,latexsym}
\usepackage{graphicx}
\usepackage{mathtools}
\usepackage{color}
\usepackage[hyphens]{url}
\usepackage[T1]{fontenc}
\usepackage{hyperref}
\usepackage{amsmath}
\usepackage{mathdots}
\usepackage{breqn}
\usepackage[toc,page]{appendix}
\usepackage{url}    %does nice formatting of URLs
\usepackage{breqn}
\usepackage{hyperref}
\usepackage{breakurl}
\newcommand{\bburl}[1]{\textcolor{blue}{\url{#1}}}

\calclayout

\makeatletter
\newcommand{\monthyear}[1]{%
  \def\@monthyear{\uppercase{#1}}}
\newcommand{\volnumber}[1]{%
  \def\@volnumber{\uppercase{#1}}}
\AtBeginDocument{%
\def\ps@plain{\ps@empty
  \def\@oddfoot{\@monthyear \hfil \thepage}%
  \def\@evenfoot{\thepage \hfil \@volnumber}}
\def\ps@firstpage{\ps@plain}
\def\ps@headings{\ps@empty
  \def\@evenhead{%
    \setTrue{runhead}%
    \def\thanks{\protect\thanks@warning}%
    \uppercase{\ }\hfil}%
  \def\@oddhead{%
    \setTrue{runhead}%
    \def\thanks{\protect\thanks@warning}%
    \hfill\uppercase{Generalized Schreier sets, linear recurrence relation, Tur\'{a}n graphs}}%
  \let\@mkboth\markboth
  \def\@evenfoot{%
    \thepage \hfil \@volnumber}%
  \def\@oddfoot{%
    \@monthyear \hfil \thepage}%
  }%
\footskip=25pt
\pagestyle{headings}%
}
\makeatother

\theoremstyle{plain}
\numberwithin{equation}{section}
\newtheorem{thm}{Theorem}[section]

\newcommand{\seqnum}[1]{\href{https://oeis.org/#1}{\underline{#1}}}

\newcommand{\ignore}[1]{}

\newcommand\be{\begin{eqnarray}}
\newcommand\ee{\end{eqnarray}}
\newcommand\bea{\begin{eqnarray}}
\newcommand\eea{\end{eqnarray}}
\newcommand\ben{\begin{enumerate}}
\newcommand\een{\end{enumerate}}

\newtheorem{lem}[thm]{Lemma}

\begin{document}
%% replace the values in the next three lines by the correct information

\monthyear{}
\volnumber{Volume, Number}
\setcounter{page}{1}
\title{Generalized Schreier sets, linear recurrence relation, Tur\'{a}n graphs}

\author{Kevin Beanland}
\author{H\`ung Vi\d{\^e}t Chu}
\author{Carrie E. Finch-Smith}

\address{Department of Mathematics, Washington and Lee University, Lexington, VA 24450, USA} \email{beanlandk@wlu.edu}
\address{Department of Mathematics, University of Illinois at Urbana-Champaign, Urbana, IL 61820} \email{hungchu2@illinois.edu}
\address{Department of Mathematics, Washington and Lee University, Lexington, VA 24450, USA}\email{finchc@wlu.edu}

\date{\today}

\begin{abstract}
We prove a linear recurrence relation for a large family of generalized Schreier sets, which generalizes the Fibonacci recurrence proved by Bird and higher order Fibonacci recurrence proved by the second author et al. Furthermore, we show a relationship between Schreier-type sets and Tur\'{a}n graphs.  
\end{abstract}

\thanks{The author is thankful for the anonymous referee's careful reading of this article.}

\maketitle
%%%%%%%%%%%%%%%%%%%%%%%%%%%%%%%%%%%%%%%%%%%%%%%%%%%%%%%%%%%%%%%%%%%%%%%%%%%%%%%%%%%%%%%%%%%%%%%%%%%%%%%%%%%%%%%%%%%%%%%%%%%%%%%%%%%%%%%%%%%%%%%%%%%%%%%%%%%%%%%%%%%%%
%%%%%%%%%%%%%%%%%%%%%%%%%%%%%%%%%%%%%%%%%%%%%%%%%%%%%%%%%%%%%%%%%%%%%%%%%%%%%%%%%%%%%%%%%%%%%%%%%%%%%%%%%%%%%%%%%%%%%%%%%%%%%%%%%%%%%%%%%%%%%%%%%%%%%%%%%%%%%%%%%%%%%
%%%%%%%%%%%%%%%%%%%%%%%%%%%%%%%%%%%%%%%%%%%%%%%%%%%%%%%%%%%%%%%%%%%%%%%%%%%%%%%%%%%%%%%%%%%%%%%%%%%%%%%%%%%%%%%%%%%%%%%%%%%%%%%%%%%%%%%%%%%%%%%%%%%%%%%%%%%%%%%%%%%%%
\section{Introduction}
A finite set $F\subset\mathbb{N}$ is said to be \textit{Schreier} if $\min F \ge |F|$, where $|F|$ is the cardinality of $F$. The namesake of Schreier sets is J\'osef Schreier who introduced these sets in the construction of a Banach space solving a problem of Banach and Saks \cite{S}. In a blog post \cite{B}, Alistair Bird showed that the Fibonacci
sequence appears if we count Schreier sets under certain conditions. In particular, if we set 
$S_{n} := \{F \subset \mathbb{N}: \min F \ge |F|\mbox{ and }\max F = n\}$, then
$|S_{1}|= 1, |S_2|= 1$, and $|S_{n+2}|= |S_{n+1}|+|S_{n}|$ for all $n \ge 1$. There has been research on generalizing Bird's result to  higher order recurrences (see \cite[Theorems 4, 5, 6]{C1} and \cite[Theorems 1.1, 1.3]{C2}) and on investigating the relationship between Schreier-type sets and partial sums of the Fibonacci and Gibonacci sequences \cite{C3, M}. 

The first main result of this paper proves a recurrence relation from a large family of generalized Schreier sets. 
For $(p,q,n) \in \mathbb{N}^3$, we define
$$S^{p/q}_n\ =\ \{ F \subset \mathbb{N} : q \min F\ge p|F|\mbox{ and } \max F=n\}.$$
Observe that $S^{p/1}_n$ is a special case considered in \cite[Theorem 1.1]{C2}.

\begin{thm}\label{mt1}
Let $(p,q) \in \mathbb{N}^2$. For $n \in \mathbb{N}$ with $n \ge p+q$, we have 
\begin{equation}\label{e10}|S^{p/q}_{n}| \ =\ \sum_{k=1}^q (-1)^{k+1}\binom{q}{k}|S^{p/q}_{n-k}| + |S^{p/q}_{n-(p+q)}|.\end{equation}
\end{thm}

If $p=q=1$, we have the Fibonacci recurrence stated above and proved by Bird. 
If $q=1$ and $p \in \mathbb{N}$, we have \cite[Theorem 1.1]{C2}
$$|S^p_{n}|\ =\ |S_{n-1}^p|+|S^p_{n-(p+1)}|.$$
The cases $q\in \mathbb{N}$ and $p=1$ are new and, in the authors' opinion,  unexpected and elegant. 
We also note that $p/q$ need not be in the simplified form. For different forms $p/q$, \eqref{e10} gives equivalent recurrences\footnote{This independence of the recurrence relation has the same spirit as \cite[Remark 1.5]{C2}, where the depth of the recurrence is independent of one of the parameters.}.

Our second result of this short note connects Schreier-type sets with Tur\'{a}n graphs. A Tur\'{a}n graph, denoted by $T(n,p)$, 
is the $n$-vertex complete $p$-partite graph whose parts differ in size by at most $1$. That is, $T(n,p)$ has $n$ vertices separated into  $p$ subsets, with sizes as equal as possible, and two vertices are connected by an edge if and only if they belong to different subsets \cite{T}. With an abuse of notation, we also write $T(n,p)$ to mean the number of edges of the corresponding graph. For each fixed $p\ge 2$, the sequence $(T(n,p))_{n=1}^\infty$ is available in OEIS \cite{Sln} (for example, see
\seqnum{A002620},
\seqnum{A000212},
\seqnum{A033436}, and
\seqnum{A033437}.)

%A set $F\subset \mathbb{N}$ is said to \textit{have no gaps} if $F$ contains all integers between $a$ and $b$ whenever $a, b\in F$. Observe that a no-gap set $F$ is uniquely determined once $\min F$ and $|F|$ are known.  {\color{red} Should we just say integer intervals?}
Define $$Sr(n,p) \ =\ \left|\left\{F\subset [n]\,:\, p\min F\ge |F|\mbox{ and }F \mbox{ is an interval}\right\}\right|,$$
where $[n] = \{1, 2, \ldots, n\}$. Notice that in contrast to the definition of $S^{p/q}_n$, the definition of $Sr(n,p)$ does not require $\max F=n$.  

\begin{thm}\label{mt2}
For all $(n,p)\in \mathbb{N}^2$ with $n\ge p$, we have
$$Sr(n, p) \ =\ T(n+1, p+1).$$
\end{thm}

\section{Proof of Theorem \ref{mt1}}

The main tool is the following lemma.

\begin{lem}
Fix $(p,q) \in \mathbb{N}^2$, $n\ge p+q$, and let $G \subset \{n-q, \ldots n-1\}$ be nonempty. Define $$A_{G}\ =\ \{F \in S^{p/q}_n: G \cap F =\emptyset\}.$$
Then $|A_G|\ =\ |S_{n-|G|}^{p/q}|$. 
\label{l3}
\end{lem}

\begin{proof}
Fix $G$ and let $\psi_G:\{1,\ldots,n \} \setminus G\to \{1,\ldots,n-|G| \}$ be the unique increasing bijection.
Define $\phi_G: A_G \to S_{n-|G|}^{p/q}$ by
$$\phi_G(F)\ =\ \{\psi_G(i): i \in F\}.$$
Showing that $\phi_G$ is a bijection is straightforward but quite technical. 

First, we show that $\phi_G$ is well-defined; that is, the range of $\phi_G$ is $S_{n-|G|}^{p/q}$. Let $F\in A_G$. By definition, $n=\max F$ and so, $n-|G|= \max \phi_G(F)$. Note that $p|\phi_G(F)|=p|F|$. If $\min F < \min G$, then $\min \phi_G(F)= \min F$, and in this case,
$$p|\phi_G(F)|\ =\ p|F|\ \le\  q\min F\ =\ q\min \phi_G(F),$$ as desired. Otherwise $\min F>\min G$. In this case, $\min \phi_G(F)\ge n-q$ and $|F|\le q$ since $G\subset \{n-1, \ldots,n-q\}$, $F \cap G=\emptyset$ and is $G$ non-empty. Since $n \ge p+q$, we have
$$p|\phi_G(F)|\ =\ p|F|\ \le\ pq \ \le\ q(n-q)\ \le\  q \min\phi_G(F).$$
This is the desired result.

The injectivity of $\phi_G$ follows immediately from the definition of $\psi_G$ and $\phi_G$. It remains to show that 
$\phi_G$ is surjective. Fix $H\in S_{n-|G|}^{p/q}$. Define 
$$F\ =\ \{\psi_G^{-1}(i): i \in H\}.$$
By definition, $\phi_G(F)=H$ and $F\cap G=\emptyset$. Note that $\max F=n$ and 
$$p|F|\ =\ p|H|\ \le\  q\min H\ =\ q\psi_G(\min F) \ \le\ q \min F.$$
This finishes the proof of the lemma.
\end{proof}

\begin{proof}[Proof of Theorem \ref{mt1}] 
Using notation from Lemma \ref{l3}, the set $S^{p/q}_{n} \setminus \cup_{i=1}^q A_{\{n-i\}}$ is 
$$A\ :=\ \{F \in S^{p/q}_n: \{n-q, \ldots, n-1\} \subset F\}.$$
We claim that $|A|=|S^{p/q}_{n-(p+q)}| $. The bijection $\phi:A \to S^{p/q}_{n-(p+q)}$ is defined by
$$\phi(F)\ =\  (F\setminus\{n-q+1, \ldots,n \})-p.$$
Note first that $n-q =\max (F\setminus\{n-q+1, \ldots,n \})$ and so $n-(p+q) = \max \phi(F)$. In addition we have
$$p|\phi(F)|\ =\ p(|F|-q)\ \le\ q\min F-pq\ =\ q(\min F-p)\ =\ q \min \phi(F).$$
Therefore $\phi(F) \in S^{p/q}_{n-(p+q)}$. To see that $\phi$ is injective is trivial. We show that $\phi$ is surjective.
Let $H \in S^{p/q}_{n-(p+q)}$ and define $F=(H+p)\cup\{n-q+1,\ldots,n\}$. Then $\phi(F)=H$ and $F \in A$, since 
$$p|F|\ =\ p(|H|+q)\ \le\  q(\min H + p)\ = \ q \min F.$$

Let $\mathcal{G}_i=\{G \subset \{n-q, \ldots, n-1\}: |G|=i\}.$ By the inclusion-exclusion principle and Lemma \ref{l3}, we obtain
\begin{equation}
    \begin{split}
|S^{p/q}_{n}| & \ =\ |A| + \sum_{G \in \mathcal{G}_1} |A_G| -  \sum_{G \in \mathcal{G}_2} |A_G| + \sum_{G \in \mathcal{G}_3} |A_G| - \cdots +(-1)^{q+1} \sum_{G \in \mathcal{G}_q} |A_G| \\
& \ =\ |S^{p/q}_{n-(p+q)}| +\binom{q}{1} |S^{p/q}_{n-1}|- \binom{q}{2} |S^{p/q}_{n-2}| +\cdots + (-1)^{q+1}\binom{q}{q} |S^{p/q}_{n-q}|.
\end{split}
\end{equation}
This is the desired result. 
\end{proof}

\section{Proof of Theorem \ref{mt2}}

The following is well-known and can, for example, be found on the Wikipedia page for Tur\'an graphs. 

\begin{lem}\label{l1}
For all $(n, p)\in \mathbb{N}^2$ with $n>p$, it holds that
\begin{equation}\label{e1}T(n, p) \ =\ \frac{p-1}{2p}(n^2-q^2) + \binom{q}{2},\end{equation}
where $q := n - p\lfloor n/p\rfloor$.
\end{lem}

%\begin{proof}
%Fix $(n,p) \in \mathbb{N}^2$. The $n$ vertices are divided into $p$ parts as equally as possible. There are $q: = n - p\lfloor n/p\rfloor$ parts with $\lfloor n/p\rfloor+1$ vertices and $p-q$ parts with $\lfloor n/p\rfloor$ vertices. Since vertices in the same part are independent and $T(n,p)$ is complete $p$-partite, the number of edges is 
%\begin{align*}
 %   T(n,p) &\ =\ \binom{n}{2} - q\binom{\lfloor n/p\rfloor+1}{2} - (p-q)\binom{\lfloor n/p\rfloor}{2}\\
  %  &\ =\ \frac{1}{2}(n^2-n-2q\lfloor n/p\rfloor + p \lfloor n/p\rfloor - p \lfloor n/p\rfloor ^2)\\
  %  &\ =\ \frac{1}{2}(n^2 - q - 2q\lfloor n/p\rfloor - p \lfloor n/p\rfloor ^2)\\
  %  &\ =\ \frac{1}{2}(n^2 - q^2 - 2q\lfloor n/p\rfloor - p \lfloor n/p\rfloor ^2) + \binom{q}{2}\\
%    &\ =\ \frac{1}{2}(p^2\lfloor n/p\rfloor^2  + 2pq\lfloor n/p\rfloor- 2q\lfloor n/p\rfloor - p \lfloor n/p\rfloor ^2) + \binom{q}{2}\\
 %   &\ =\ \frac{p-1}{2p}(2qp\lfloor n/p\rfloor+p^2\lfloor n/p\rfloor^2)+ \binom{q}{2}\ =\ \frac{p-1}{2p}(n^2-q^2) + \binom{q}{2}.
%\end{align*}
%\end{proof}

\begin{lem}\label{l2}We use $a \wedge b$ to indicate $\min \{a, b\}$. 
For all $(n,p)\in \mathbb{N}^2$, it holds that 
\begin{align}\label{e2}Sr(n,p) &\ =\  \sum_{m=1}^n pm \wedge (n+1-m) \\
&\ =\ \label{e3}\begin{cases}1 &\mbox{ if } n=1,\\ \binom{n+1}{2} &\mbox{ if } p > n \ge 2,\\ \frac{1}{2}\left(p(\Delta+1)\Delta + (n-\Delta+1)(n-\Delta)\right) &\mbox{ if }p\le n,\end{cases}\end{align}
where $\Delta = \lfloor (n+1)/(p+1)\rfloor$.
\end{lem}

\begin{proof}
We build a set $F\subset [n]$ that satisfies: 1) $p\min F\ge |F|$ and 2) $F$ is an interval. We denote the smallest element of $F$ by $m$, which can be chosen from $1$ to $n$. Once $m$ is fixed, we choose $c: = |F|$, which must satisfy $c \le pm$ and $c+m-1\le n$. The latter condition is to guarantee that $\max F\le n$. Once $m$ and $c$ are chosen, then $F$ is unique since it is an interval.
We obtain the formula for $Sr(n,p)$
\begin{equation*}
   Sr(n,p)\ =\  \sum_{m=1}^n \sum_{c=1}^{pm\wedge (n+1-m)}1\ =\ \sum_{m=1}^n pm \wedge (n+1-m),
\end{equation*}
which is \eqref{e2}. 

We now derive \eqref{e3}:
\begin{enumerate}
    \item By \eqref{e2}, $Sr(1, p) = p\wedge 1 = 1$.
    \item When $p > n\ge 2$, we have $pm\wedge (n+1-m) = n+1-m$ and so,
$$Sr(n,p) \ =\ \sum_{m=1}^n (n+1-m)\ =\ \binom{n+1}{2}.$$
    \item When $p \le n$, we have
    \begin{align*}
        \sum_{m=1}^n pm \wedge (n+1-m)&\ =\ \sum_{m=1}^{\Delta} pm  + \sum_{m=\Delta + 1}^{n+1} (n+1-m)\\
        &\ =\ \frac{1}{2}p(1+\Delta)\Delta + \frac{1}{2}(n-\Delta)(n-\Delta+1).
    \end{align*}
\end{enumerate}
This proves \eqref{e3}.
\end{proof}

\begin{proof}[Proof of Theorem \ref{mt2}]
We prove $Sr(n, p) \ =\ T(n+1, p+1)$ for all $(n,p)\in \mathbb{N}^2$ with $n\ge p$. If $n = p$, then by definitions, $T(n+1, p+1) = \binom{n+1}{2}$, and $Sr(n,p) = n+(n-1)+\cdots+1 = \binom{n+1}{2}$.
If $n>p$, by Lemmas \ref{l1} and \ref{l2}, we want to show that 
$$\frac{p((n+1)^2-q^2)}{2(p+1)} + \binom{q}{2}\ =\ \frac{1}{2}\left(p(\Delta+1)\Delta + (n-\Delta+1)(n-\Delta)\right),$$
which is equivalent to 
\begin{align}\label{e5}p\Delta(2n+2-(p+1)\Delta) + (n+1-(p+1)\Delta)&(n-(p+1)\Delta)\nonumber \\
&\ =\ p\Delta(\Delta+1) + (n-\Delta+1)(n-\Delta).\end{align}
Simple algebraic manipulation of the three variables $p, n$, and $\Delta$ confirms that two sides of \eqref{e5} are equal. This completes our proof. 
\end{proof}

It would be interesting to see a proof of Theorem \ref{mt2} which gives an explicit bijection between the edges of $T(n+1,p+1)$ and the elements of the set $$\left\{F\subset [n]\,:\, p\min F\ge |F|\mbox{ and }F \mbox{ is an interval}\right\}.$$

\ \\

\noindent MSC2020: 11B39 (primary), 11B37 (secondary)

\end{document}